\documentclass[11pt]{amsart}       

\usepackage{graphicx}
\usepackage{latexsym}
\usepackage{amssymb}
\usepackage{amsmath}

\newtheorem{theorem}{Theorem}[section]
\newtheorem{lemma}[theorem]{Lemma}
\newtheorem{definition}[theorem]{Definition}
\newtheorem{claim}[theorem]{Claim}

\newcommand{\clubsucc}[2]{\mathrm{succ}_{#1}(#2)}
\newcommand{\seq}[2]{\left\langle #1\;|\; #2\right\rangle}
\DeclareMathOperator{\Ad}{Ad}
\DeclareMathOperator{\cf}{cf}
\DeclareMathOperator{\ran}{ran}

\begin{document}

\title{
Some Calkin algebras have outer automorphisms
}
\thanks{This work was initiated at the Mittag-Leffler Institute during the authors' visit in September, 2009.  The first author was partially supported by NSERC}

\author{Ilijas Farah}
\address{Ilijas Farah:
Department of Mathematics and Statistics, York University,
4700 Keele Street, North York, Ontario, Canada, M3J 1P3; and Matematicki
Institut, Kneza Mihaila 34, Belgrade, Serbia.}

\email{ifarah@mathstat.yorku.ca}           
\author{Paul McKenney}

\address{Paul McKenney: Department of Mathematical Sciences, 
Carnegie Mellon University, Pittsburgh, USA}
\email{pmckenne@andrew.cmu.edu}

\author{Ernest Schimmerling}

\address{Ernest Schimmerling: Department of Mathematical Sciences, 
Carnegie Mellon University, Pittsburgh, USA}
\email{eschimme@andrew.cmu.edu}

\maketitle

\begin{abstract}
We consider various quotients of the C*-algebra of bounded operators on a 
nonseparable Hilbert space, and prove in some cases that, assuming some restriction
of the Generalized Continuum Hypothesis, there are many outer automorphisms.
\end{abstract}

\section{Introduction}
\label{intro}
  Let $\mathcal{H}$ be a Hilbert space.  The \emph{Calkin algebra} over $\mathcal{H}$ is the quotient $\mathcal{C}(\mathcal{H}) = \mathcal{B}(\mathcal{H})/\mathcal{K}$, where $\mathcal{B}(\mathcal{H})$ is the C*-algebra of bounded, linear operators on $\mathcal{H}$, and $\mathcal{K}$ is its ideal of compact operators.  Assuming the Continuum Hypothesis, Phillips and Weaver constructed $2^{2^{\aleph_0}}$-many automorphisms of the Calkin algebra on the Hilbert space of dimension $\aleph_0$ (\cite{PW}).  Since there are only $2^{\aleph_0}$-many automorphisms of $\mathcal{C}(\mathcal{H})$ which are \emph{inner} (that is, implemented by conjugation by a unitary), this implies in particular that there are many more outer automorphisms than there are inner ones, in the presence of CH.  
  
  The first author proved in~\cite{F.C} that it is relatively consistent with ZFC that all automorphisms of the Calkin algebra on a separable Hilbert space are inner.  This establishes the existence of an outer automorphism as a question independent of ZFC.  The assumption made there was \emph{Todor\v cevi\'c's Axiom} (TA), a combinatorial principle also known as the \emph{Open Coloring Axiom}.  TA has a number of consequences in other areas of mathematics, and follows from the \emph{Proper Forcing Axiom} (PFA), which is itself well-known for its influence on certain kinds of rigidity in mathematics (see~\cite{M.PFA}).  The first author extended this result to prove that all automorphisms of the Calkin algebra over \emph{any} Hilbert space, separable or not, are inner, assuming PFA (\cite{F.AC}).

  The development of these results parallels those in the study of the automorphisms of the Boolean algebra $\mathcal{P}(\omega)/\mathrm{fin}$.  Rudin (\cite{R}) discovered early on that, assuming CH, there are many automorphisms of $\mathcal{P}(\omega)/\mathrm{fin}$ that are not \emph{trivial}, i.e. induced by functions $e : \omega\to\omega$; Shelah (\cite{Sh.PF}) much later proved the consistency of the opposite result, that all automorphisms are trivial.  Shelah and Stepr\-ans then showed that all automorphisms are trivial assuming PFA (\cite{ShStep}), and then Veli\v ckovi\'c showed using PFA that all automorphisms of $\mathcal{P}(\kappa)/\mathrm{fin}$ are trivial, for every infinite cardinal $\kappa$ (along with reducing the assumption to $\mathrm{TA} + \mathrm{MA}_{\aleph_1}$ in the original case $\kappa = \omega$).

  One might ask for the consistency of outer automorphisms of $\mathcal{C}(\mathcal{H})$ when $\mathcal{H}$ is nonseparable, or nontrivial automorphisms of $\mathcal{P}(\kappa)/\mathrm{fin}$ when $\kappa$ is uncountable.  The latter result is easy, though for trivial reasons, since the automorphisms of $\mathcal{P}(\omega)/\mathrm{fin}$ can all be extended to automorphisms of $\mathcal{P}(\kappa)/\mathrm{fin}$, and any extension of a nontrivial automorphism of $\mathcal{P}(\omega)/\mathrm{fin}$ must also be nontrivial.  In the case of $\mathcal{C}(\mathcal{H})$ this is not so clear, and in fact it is not yet known whether the existence of an outer automorphism of $\mathcal{C}(\mathcal{H})$, when $\mathcal{H}$ is nonseparable, is consistent with ZFC.  However in the case where $\mathcal{H}$ is nonseparable there is more than one quotient of $\mathcal{B}(\mathcal{H})$ to consider.  In this note we study some of these different quotients, and offer some alternatives results;

  \begin{theorem}
    \label{large.intro}
    Let $\mathcal{H}$ be a Hilbert space of some regular, uncountable dimension $\kappa$ and let $\mathcal{J}$ be the ideal in $\mathcal{B}(\mathcal{H})$ of operators whose range has dimension less than $\kappa$.  If $2^\kappa = \kappa^+$, then the quotient $\mathcal{B}(\mathcal{H})/\mathcal{J}$ has $2^{\kappa^+}$-many outer automorphisms.
  \end{theorem}

  \begin{theorem}
    \label{small.intro}
    Let $\mathcal{H}$ be a Hilbert space of dimension $\aleph_1$, let $\mathcal{J}$ be the ideal of operators on $\mathcal{H}$ whose range has dimension $< \aleph_1$, and let $\mathcal{K}$ be the ideal of compact operators.  If CH holds, then $\mathcal{J}/\mathcal{K}$ has $2^{\aleph_1}$-many outer automorphisms.
  \end{theorem}

  Theorem~\ref{large.intro} is perhaps most striking in the case $\kappa = \aleph_1$, for in this case its only set-theoretic assumption, $2^{\aleph_1} = \aleph_2$, follows already from PFA.  Hence in a model of PFA, there are many outer automorphisms of $\mathcal{B}(\mathcal{H}) / \mathcal{J}$, and yet no outer automorphisms of $\mathcal{B}(\mathcal{H})/\mathcal{K}$.
  

  Our notation is mostly standard.  All Hilbert spaces considered are complex Hilbert spaces.  When $\mathcal{H}$ is a Hilbert space, $\mathcal{B}(\mathcal{H})$ denotes the C*-algebra of bounded linear operators from $\mathcal{H}$ to $\mathcal{H}$, $\mathcal{K}(\mathcal{H})$ denotes the closed $*$-ideal in $\mathcal{B}(\mathcal{H})$ given by the compact operators on $\mathcal{H}$, and $\mathcal{J}(\mathcal{H})$ denotes the $*$-ideal of operators whose range has dimension strictly less than the dimension of $\mathcal{H}$.  When the Hilbert space $\mathcal{H}$ is understood we will often drop it in our notation and just use $\mathcal{B},\mathcal{K}$, and $\mathcal{J}$.  Note that when $\mathcal{H}$ is nonseparable, $\mathcal{J}$ is already norm-closed, and
  \[
    \mathcal{K} \subset \mathcal{J} \subset \mathcal{B}
  \]
  If $x\in\mathcal{B}$ then we will use $[x]_{\mathcal{K}}$ and $[x]_{\mathcal{J}}$ to denote the quotients of $x$ by $\mathcal{K}(\mathcal{H})$ and $\mathcal{J}(\mathcal{H})$ respectively.  When $A$ is a set, we will write $\ell^2(A)$ for the Hilbert space of square-summable functions $\xi : A\to\mathbb{C}$.  We will also often write $\mathcal{B}(\mathcal{H}) = \mathcal{B}_A$, $\mathcal{J}(\mathcal{H}) = \mathcal{J}_A$, and $\mathcal{K}(\mathcal{H}) = \mathcal{K}_A$ when $\mathcal{H} = \ell^2(A)$.  When $A\subseteq B$ we will identify $\ell^2(A)$ with a closed subspace of $\ell^2(B)$ in the obvious way.  Finally, if $A$ is a C*-algebra and $x$ is an element of $A$ then $\Ad{x} : A\to A$ is the map $a\mapsto xax^*$.  When $A$ has a multiplicative unit and $x$ is a unitary element of $A$, i.e. $x^*x = xx^* = 1_A$, then $\Ad{x}$ is an automorphism of $A$.

  \section{Large ideals}
  \label{sec:large}

  In this section we prove Theorem~\ref{large.intro}.  Before beginning the proof we will need some notation;
  \begin{definition}
    If $C$ is club in $\kappa$, we define
    \[
      x\in\mathcal{D}[C] \iff \forall \alpha\in C\quad\mbox{$\ell^2(\alpha)$ is an invariant subspace of $x$ and $x^*$}
    \]
  \end{definition}
  Note that $\mathcal{D}[C]$ is a C*-subalgebra of $\mathcal{B}_\kappa$, and in fact is a von Neumann subalgebra of $\mathcal{B}_\kappa$, though we will not use this latter fact.  We also set down some convenient notation for the successor of an ordinal in a club;
  \begin{definition}
    If $C$ is club in $\kappa$ and $\alpha\in C$, then $\clubsucc{C}{\alpha}$ denotes the minimal element of $C$ strictly greater than $\alpha$.
  \end{definition}
  Note that if $C$ is club in $\kappa$, then we have in fact
  \[
    x\in\mathcal{D}[C] \iff \forall \alpha\in C\quad\mbox{$\ell^2([\alpha,\clubsucc{C}{\alpha}))$ is an invariant subspace of $x$}
  \]
  Finally, if $A,B\subseteq\kappa$ then we write $A\subseteq^* B$ if and only if $|A\setminus B| < \kappa$.

  \begin{lemma}
    \label{covering}
    For every $x\in\mathcal{B}_\kappa$, there is some club $C$ in $\kappa$ such that $x\in\mathcal{D}[C]$.
  \end{lemma}

  \begin{proof}
    Let $\theta$ be large and regular, and let $M_\alpha$, for $\alpha < \kappa$, be a club of elementary substructures of $H(\theta)$, each of size $< \kappa$, and with $x$ and $\ell^2(\kappa)$ in $M_0$.  Then if $\delta = \sup(M_\alpha\cap\kappa)$, we clearly have that $\ell^2(\delta)$ is an invariant subspace of $x$, and such ordinals $\delta$ make up a club in $\kappa$.
  \qed\end{proof}

  \begin{lemma}
    If $C\subseteq^* \widetilde{C}$ are clubs in $\kappa$, then $\mathcal{D}[\widetilde{C}] \subseteq_{\mathcal{J}} \mathcal{D}[C]$, by which we mean
    \[
      \forall x\in\mathcal{D}[\widetilde{C}]\;\exists y\in\mathcal{D}[C]\quad x - y\in\mathcal{J}
    \]
  \end{lemma}
  \begin{proof}
    If $\gamma < \kappa$ is such that $C\cap [\gamma,\kappa) \subseteq \widetilde{C}$, then for every $\delta\in \widetilde{C}$,
    \[
      \delta\ge \gamma\implies [\delta,\clubsucc{\widetilde{C}}{\delta}) \subseteq [\delta,\clubsucc{C}{\delta})
    \]
    Thus if $x\in\mathcal{D}[\widetilde{C}]$, we see that $PxP\in\mathcal{D}[C]$, where $P$ is the projection onto the subspace $\ell^2([\gamma,\kappa))$.
  \qed\end{proof}

  \begin{lemma}
    \label{key}
    Let $C$ be club in $\kappa$ and let $u$ and $v$ be unitary operators on $\ell^2(\kappa)$, which are diagonal with respect to the standard basis; say $f,g :\kappa\to\mathbb{T}$ are the diagonal values of $u$ and $v$ respectively.  Then $\Ad{[u]_{\mathcal{J}}}$ and $\Ad{[v]_{\mathcal{J}}}$ agree on $\mathcal{D}[C]/\mathcal{J}$ if and only if there is some $\epsilon < \kappa$ such that the map
    \[
      \xi\mapsto \frac{f(\xi)}{g(\xi)} = f(\xi)\overline{g(\xi)}
    \]
    is constant on each interval of the form $[\delta,\clubsucc{C}{\delta})$ with $\delta\in C\cap [\epsilon,\kappa)$.
  \end{lemma}
  \begin{proof}
    Let $h(\xi) = f(\xi)\overline{g(\xi)}$ for each $\xi < \kappa$.  We will write $(*)$ for the condition
    \[
      \exists\epsilon\;\forall\delta\in C\quad \delta\ge \epsilon \implies \mbox{$h$ is constant on the interval $[\delta,\clubsucc{C}{\delta})$}
    \]
    as in the conclusion of the lemma.  Now, note that $u$ and $v$ are trivially in the algebra $\mathcal{D}[C]$.  The following are equivalent;
    \begin{enumerate}
      \item  $\Ad{[u]_{\mathcal{J}}}$ and $\Ad{[v]_{\mathcal{J}}}$ agree on $\mathcal{D}[C]/\mathcal{J}$,
      \item  for each $x\in\mathcal{D}[C]$, $uxu^* - vxv^*$ is in $\mathcal{J}$,
      \item  \label{center} $[v^*u]_{\mathcal{J}}$ is in the center of the algebra $\mathcal{D}[C]/\mathcal{J}$.
    \end{enumerate}
    We will show that condition~\eqref{center} holds if and only if $(*)$ holds.  First suppose $(*)$ does not hold; then there is an unbounded subset $A$ of $C$, and sequences $\sigma_\delta$,$\tau_\delta$ indexed by $\delta\in A$, such that for each $\delta\in A$, $\delta\le \sigma_\delta < \tau_\delta < \clubsucc{C}{\delta}$ and $h(\sigma_\delta) \neq h(\tau_\delta)$.  Let $x$ be the operator defined by
    \[
      x(e_\alpha) =
        \left\{
          \begin{array}{ll}
            e_{\sigma_\delta} & \mbox{$\alpha = \tau_\delta$ for some $\delta\in A$} \\
            e_{\tau_\delta} & \mbox{$\alpha = \sigma_\delta$ for some $\delta\in A$} \\
            0 & \mbox{otherwise}
          \end{array}
        \right.
    \]
    Then $x\in\mathcal{D}[C]$, and for each $\delta\in A$,
    \[
      (v^*u x)e_{\sigma_\delta} = h(\tau_\delta)e_{\tau_\delta}\qquad (x v^*u)e_{\sigma_\delta} = h(\sigma_\delta)e_{\tau_\delta}
    \]
    It follows that $v^*u x - x v^*u$ is not in the ideal $\mathcal{J}$, so condition~\eqref{center} does not hold.  Now suppose $(*)$ does hold, and choose $\epsilon$ as in this condition.  If $x\in\mathcal{D}[C]$, then for all $\alpha \ge \epsilon$, if $\alpha\in [\delta,\clubsucc{C}{\delta})$ where $\delta\in C$ then we have
    \[
      (v^*u x)e_\alpha = h(\alpha) x e_\alpha = (x v^* u)e_\alpha
    \]
    and it follows that $P(v^* u)P$ is in the center of $\mathcal{D}[C]$, where $P$ is the projection onto $\ell^2([\epsilon,\kappa))$.
  \qed\end{proof}

  We are now ready to prove Theorem~\ref{large.intro}.
  \begin{proof}
    Let $\seq{E_\alpha}{\alpha \in\lim(\kappa^+)}$ enumerate the clubs in $\kappa$.  We will construct a sequence of clubs $C_s$ in $\kappa$, and functions $f_s :\kappa\to\mathbb{T}$, indexed by $s\in 2^{<\kappa^+}$, such that
    \begin{enumerate}
      \item  \label{coherence.clubs}
        If $s\subset t$, then $C_t \subseteq^* C_s$.
      \item  \label{coherence.unitaries}
        If $s\subset t$, then there is an $\epsilon < \kappa$ such that for every $\delta\in C_s$ with $\delta \ge \epsilon$, the function $f_s\overline{f_t}$ is constant on the interval $[\delta,\clubsucc{C_s}{\delta})$.
      \item  \label{branching}
        For all $s$, $C_{s^\smallfrown 0} = C_{s^\smallfrown 1}$, and for unboundedly many $\delta\in C_{s^\smallfrown 0} = C_{s^\smallfrown 1} = C$, the function $f_{s^\smallfrown 0}\overline{f_{s^\smallfrown 1}}$ is not constant on $[\delta,\clubsucc{C}{\delta})$.
      \item  \label{cofinality}
        If $s$ has length some limit ordinal $\alpha < \kappa^+$, then $C_s \subseteq^* E_\alpha$.
    \end{enumerate}

    \begin{claim}
      This suffices.
    \end{claim}
    \begin{proof}
      For each $s\in 2^{<\kappa^+}$, let $u_s$ be the diagonal unitary in $\mathcal{B}_\kappa$ with diagonal elements given by $f_s$.  For each $\zeta\in 2^{\kappa^+}$, and $x\in \mathcal{D}[C_{\zeta\upharpoonright \alpha}]$, define
      \[
        \Phi_\zeta ([x]) = [u_{\zeta\upharpoonright\alpha} x u_{\zeta\upharpoonright\alpha}^*]
      \]
      By~\eqref{coherence.clubs}, \eqref{coherence.unitaries}, and Lemma~\ref{key}, $\Phi_\zeta$ is well-defined on the union of the algebras $\mathcal{D}[C_{\zeta\upharpoonright\alpha}]/\mathcal{J}$, over $\alpha < \kappa^+$; and by~\eqref{cofinality}, and Lemma~\ref{covering}, it follows that $\Phi_\zeta$ is defined on all of $\mathcal{B}_\kappa / \mathcal{J}$.  Since on each $\mathcal{D}[C_{\zeta\upharpoonright\alpha}]$, $\Phi_\zeta$ agrees with $\Ad{[u_{\zeta\upharpoonright\alpha}]}$, $\Phi_\zeta$ is also an injective homomorphism.  Similar arguments show that $\Phi_\zeta^{-1}$ is defined on all of $\mathcal{B}_\kappa / \mathcal{J}$, and hence $\Phi_\zeta$ is an automorphism of this quotient algebra.  Finally, if $\zeta$ and $\eta$ are distinct members of $2^{\kappa^+}$, then by~\eqref{branching} and Lemma~\ref{key} we see that $\Phi_\zeta$ and $\Phi_\eta$ are distinct automorphisms.
    \qed\end{proof}

    We construct $C_s$ and $f_s$ by induction on the length of $s\in 2^{<\kappa^+}$.  It is useful to note that all the functions $f_s$ constructed in the following actually have range contained in $\{-1,+1\}$; when proving~\eqref{coherence.unitaries} and~\eqref{branching}, then, we will drop all mention of the conjugation.  In the base case we simply set $C_{\langle\rangle} = \kappa$ and $f_{\langle\rangle}(\alpha) = 1$ for all $\alpha < \kappa$.  For the successor case, let $s\in 2^{<\kappa^+}$ be given.  Set $C_{s^\smallfrown 0} = C_{s^\smallfrown 1} = \lim(C_s)$, $f_{s^\smallfrown 0} = f_s$, and
    \[
      f_{s^\smallfrown 1}(\alpha) =
      \left\{\begin{array}{ll}
        -f_s(\alpha) & \mbox{if there is $\delta\in \lim(C_s)$ such that $\delta\le \alpha < \clubsucc{C_s}{\delta}$} \\
        +f_s(\alpha) & \mbox{otherwise} \\
      \end{array}\right.
    \]
    Obviously, the function $f_s f_{s^\smallfrown 0} = f_s^2$ is constant on each interval of $C_s$ (in fact it is constant on all of $\kappa$).  The same holds for the function $f_s f_{s^\smallfrown 1}$; if $\delta\in\lim(C_s)$ then this function has a constant value of $-1$ on all of $[\delta,\clubsucc{C_s}{\delta})$, whereas if $\delta\in C_s\setminus\lim(C_s)$ then it has a constant value of $+1$ on this interval.  Hence condition~\eqref{coherence.unitaries} is satisfied in the inductive step.  As for condition~\eqref{branching}, we note that for every $\delta\in\lim(C_s)$, the function $f_{s^\smallfrown 0} f_{s^\smallfrown 1}$ is not constant on the interval $[\delta, \clubsucc{\lim(C_s)}{\delta})$, since this function has a value of $-1$ at $\delta$ and a value of $+1$ at $\clubsucc{C_s}{\delta} < \clubsucc{\lim(C_s)}{\delta}$.  It remains to consider the limit case.  Let $s\in 2^{<\kappa^+}$ be given, and let $\alpha$ be the length of $s$.  For $\beta < \alpha$, write $f_\beta = f_{s\upharpoonright\beta}$ and $C_\beta = C_{s\upharpoonright\beta}$.  By the inductive hypothesis, for every $\beta < \gamma < \alpha$ there is an $\epsilon < \kappa$ such that
    \[
      \forall \delta\in C_\beta\quad \delta \ge \epsilon\implies \mbox{$f_\beta f_\gamma$ is constant on the interval $[\delta,\clubsucc{C_\beta}{\delta})$}
    \]
    Let $\epsilon_\beta^\gamma$ be the minimal $\epsilon\in C_\beta$ satisfying the above.  We will define $f_s$ and $C_s$ in two different ways based on the cofinality of $\alpha$.  First, suppose $\theta = \cf{\alpha} < \kappa$, and let $\alpha_\eta$, for $\eta < \theta$, be an increasing and continuous sequence of ordinals which is cofinal in $\alpha$.  Define
    \[
      C_s = \left(\bigcap_{\eta < \theta} C_{\alpha_\eta}\right)\cap E_\alpha
    \]
    It remains to define $f_s$ and show that condition~\ref{coherence.unitaries} holds.  Choose a uniform ultrafilter $\widetilde{\mathcal{U}}$ over $\theta$, and let $\mathcal{U}_\alpha$ be the ultrafilter over $\alpha$ defined in the usual way from $\widetilde{\mathcal{U}}$ using the sequence $\seq{\alpha_\eta}{\eta < \theta}$.  Now for each $\xi < \kappa$ define
    \[
      f_s(\xi) = \lim_{\beta\in\mathcal{U}_\alpha} f_\beta(\xi)
    \]
    \begin{claim}
      For every $\beta < \alpha$, $f_\beta f_s$ is constant on each interval of a tail of intervals from $C_\beta$.
    \end{claim}
    \begin{proof}
      Fix $\beta < \alpha$, and let $\epsilon = \sup_{\eta < \theta} \epsilon_\beta^{\alpha_\eta} \in C_\beta$.  Let $\delta\in C_\beta$ be given, and suppose $\delta \ge \epsilon$, but that $f_\beta f_s$ is \emph{not} constant on $[\delta,\clubsucc{C_\beta}{\delta})$; fix witnesses $\sigma < \tau$ in this interval, and say without loss of generality that $f_\beta(\sigma)f_s(\sigma) = +1$ but $f_\beta(\tau)f_s(\tau) = -1$.  By the definition of $f_s$, there are $A_0,A_1\in\mathcal{U}_\alpha$ such that
      \begin{align*}
        \forall \gamma\in A_0 & \quad f_\beta(\sigma)f_\gamma(\sigma) = +1 \\
        \forall \gamma\in A_1 & \quad f_\beta(\tau)f_\gamma(\tau) = -1
      \end{align*}
      Then if $\gamma\in A_0\cap A_1$ is larger than $\beta$ we have $f_\beta(\sigma)f_\gamma(\sigma) = +1$ and $f_\beta(\tau)f_\gamma(\tau) = -1$.  By definition of $\mathcal{U}_\alpha$ we may choose such a $\gamma$ with $\gamma = \alpha_\eta$ for some $\eta < \theta$.  But this contradicts the choice of $\epsilon_\beta^\gamma$, since $\delta \ge \epsilon > \epsilon_\beta^{\alpha_\eta}$.
    \qed\end{proof}

    Now consider the case where $\cf{\alpha} = \kappa$.  Let $\alpha_\eta$, $\eta < \kappa$, be a continuous, increasing sequence of ordinals which is cofinal in $\alpha$.  Put
    \[
      C_s = \Bigl(\underset{\eta < \kappa}{\Delta} C_{\alpha_\eta}\Bigr)\cap E_\alpha
    \]
    Again, it remains only to define $f_s$ and show that condition~\eqref{coherence.unitaries} holds.  For this we define, for $\xi < \eta$,
    \[
      \rho_\xi^\eta = \min(C_{\alpha_\xi}\setminus (\xi\cup \epsilon_{\alpha_\xi}^{\alpha_\eta}))
    \]
    and
    \[
      \epsilon(\eta) = \sup_{\xi < \eta} \rho_\xi^\eta
    \]
    Note that $\epsilon(\eta)$ is in $C_{\alpha_\xi}$ for each $\xi < \eta$.  Define $f_s(\zeta) = f_{\alpha_\eta}(\zeta)$ whenever $\epsilon(\eta) \le \zeta < \epsilon(\eta+1)$ for some $\eta < \kappa$, that is,
    \[
      f_s = \bigcup_{\eta < \kappa} f_{\alpha_\eta}\upharpoonright [\epsilon(\eta),\epsilon(\eta+1))
    \]
    \begin{claim}
      For every $\beta < \alpha$, $f_\beta f_s$ is constant on a tail of intervals from $C_\beta$.
    \end{claim}
    \begin{proof}
      We will first prove that $f_{\alpha_\xi} f_s$ is constant on a tail of intervals from $C_{\alpha_\xi}$, for each $\xi < \kappa$.  Let $\epsilon = \epsilon(\xi+1)$; then if $\delta \in C_{\alpha_\xi}$ and $\delta \ge \epsilon$, we have $\epsilon(\eta) \le \delta < \clubsucc{C_{\alpha_\xi}}{\delta} \le \epsilon(\eta+1)$ for some $\eta > \xi$.  Hence $f_s$ is equal to $f_{\alpha_\eta}$ on the interval $[\delta,\clubsucc{C_{\alpha_\xi}}{\delta})$.  Since $\delta \ge \epsilon(\eta) \ge \epsilon_{\alpha_\xi}^{\alpha_\eta}$, we see that $f_{\alpha_\xi} f_s$ is constant on this interval, as required.

      Now let $\beta < \alpha$ be given, and choose $\xi < \kappa$ such that $\beta < \alpha_\xi$.  By the above, there is an $\epsilon_0$ such that $f_{\alpha_\xi} f_s$ is constant on each interval of $C_{\alpha_\xi}$ beyond $\epsilon_0$.  Let $\epsilon_1 = \epsilon_\beta^{\alpha_\xi}$, and choose an $\epsilon_2$ such that $C_{\alpha_\xi}\cap [\epsilon_2,\kappa) \subseteq C_\beta$.  It follows that with $\epsilon = \max\{\epsilon_0,\epsilon_1,\epsilon_2\}$ we have
      \[
        \forall \delta\in C_\beta \quad \delta \ge \epsilon \implies \mbox{$f_\beta f_s$ is constant on the interval $[\delta,\clubsucc{C_\beta}{\delta})$}
      \]
    \qed\end{proof}
    Thus we have proven condition~\eqref{coherence.unitaries} in this case, and this finishes the proof of the theorem.
  \qed\end{proof}

  \section{Small ideals}
  \label{sec:small}

  In this section we work with the Hilbert space $\mathcal{H} = \ell^2(\omega_1)$.  Hence the ideals of $\mathcal{B}(\mathcal{H})$ are exactly
  \[
    0 \subset \mathcal{K} \subset \mathcal{J} \subset \mathcal{B}
  \]
  Letting $\mathcal{C}(\mathcal{L})$ denote the usual Calkin algebra over $\mathcal{L}$, i.e. $\mathcal{B}(\mathcal{L}) / \mathcal{K}(\mathcal{L})$, it follows that
  \[
    \mathcal{J}/\mathcal{K} = \bigcup_{\alpha < \omega_1} \mathcal{C}(\ell^2(\alpha)) \subset \mathcal{C}(\ell^2(\omega_1))
  \]
  We will shortly prove Theorem~\ref{small.intro}, in a slightly stronger form; namely, assuming CH, there is an automorphism $\Psi$ of the quotient $\mathcal{J}/\mathcal{K}$ whose restriction to each subalgebra $\mathcal{C}(\ell^2(\alpha))$ is an outer automorphism.  It follows also that $\Psi$ cannot be the restriction of an inner automorphism of $\mathcal{B}/\mathcal{K}$.  Before we start, we will need a special case of Lemma~4.1 from~\cite{F.C}.  We include its proof here for completeness.
  \begin{lemma}
    \label{small lemma}
    Let $\Phi$ be an automorphism of $\mathcal{C}(\mathcal{H})$, where $\mathcal{H}$ is any Hilbert space.  Then $\Phi$ is inner if and only if it is inner on some subspace $\mathcal{L}$ of $\mathcal{H}$ of the same dimension.
  \end{lemma}
  \begin{proof}
    Let $\mathcal{L}$ be a subspace of $\mathcal{H}$ of the same dimension.  Then there is an isometry $U : \mathcal{H}\to \mathcal{L}$; let $u$ be its image in $\mathcal{C}(\mathcal{H})$.  Suppose $\Phi$ is implemented by conjugation by $v$ on $\mathcal{C}(\mathcal{L})$; then for any $x\in \mathcal{C}(\mathcal{H})$,
    \[
      \Phi(x) = \Phi(u^* u x u^* u) = \Phi(u)^* v \Phi(u) x \Phi(u)^* v^* \Phi(u)
    \]
    and hence $\Phi$ is implemented by conjugation by $\Phi(u)^* v \Phi(u)$ on all of $\mathcal{C}(\mathcal{H})$.
  \qed\end{proof}

  \begin{theorem}
    \label{J mod K}
    Assume CH.  Then there are $2^{\aleph_1}$-many outer automorphisms of $\mathcal{J}/\mathcal{K}$.  Moreover, each of these automorphisms is outer in a strong sense, namely each is outer when restricted to any $\mathcal{C}(\ell^2(\alpha))$, $\alpha < \omega_1$.
  \end{theorem}
  \begin{proof}
    Let $\Phi$ be an automorphism of $\mathcal{C}(\ell^2(\omega))$.  Let $f_\alpha : \alpha\to\omega$, $\alpha < \omega_1$, be a sequence of injections satisfying
    \begin{equation}
      \label{inj.coherence}  \forall \alpha < \beta < \omega_1\quad f_\beta\upharpoonright\alpha =^* f_\alpha
    \end{equation}
    for every $\alpha < \beta < \omega_1$.  Set $A_\alpha = \ran(f_\alpha)$, let $U_\alpha : \ell^2(\alpha)\to\ell^2(A_\alpha)$ be the unitary operator induced by $f_\alpha$, and let $u_\alpha$ be its image in $\mathcal{C}(\ell^2(\omega_1))$.  Let $\Psi$ be the unique automorphism of $\mathcal{J}/\mathcal{K}$ such that
    \[
      \forall\alpha < \omega_1\quad (\Ad{u_\alpha})\circ \Psi = \Phi\circ (\Ad{u_\alpha^*})
    \]
    Condition~\eqref{inj.coherence} ensures that such a $\Psi$ exists, and verifying that $\Psi$ is an automorphism is trivial.  Lemma~\ref{small lemma} implies that if $\Phi$ is outer, then $\Phi$ is also outer on every $\ell^2(A_\alpha)$, and hence $\Psi$ is outer on every $\ell^2(\alpha)$.  By the main theorem of~\cite{PW}, there are $2^{\aleph_1}$-many outer automorphisms of $\mathcal{C}(\ell^2(\omega))$, and hence $2^{\aleph_1}$-many outer automorphisms of $\mathcal{J}/\mathcal{K}$.
    
  \qed\end{proof}

\bibliography{corona}{}
\bibliographystyle{amsplain}

\end{document}